\documentclass[11pt,twoside,reqno]{amsart}
\allowdisplaybreaks
\usepackage{amsmath,amstext,amssymb,epsfig,multicol,enumerate}
\usepackage{graphicx}
\usepackage{mathrsfs}
\calclayout
\makeatletter
\renewcommand{\@seccntformat}[1]{\bf\csname the#1\endcsname.}
\renewcommand{\section}{\@startsection{section}{1}
	\z@{.7\linespacing\@plus\linespacing}{.5\linespacing}
	{\normalfont\upshape\bfseries\centering}}
\renewcommand{\@biblabel}[1]{\@ifnotempty{#1}{#1.}}
\makeatother
\theoremstyle{plain}
\newtheorem{thm}{Theorem}[section]
\newtheorem{lem}[thm]{Lemma}

\newtheorem{cor}[thm]{Corollary}

\theoremstyle{definition}

\newtheorem{defn}[thm]{Definition}
\newtheorem{rem}{Remark}[section]

\usepackage{cancel}
\usepackage[parfill]{parskip}
\usepackage[german]{varioref}
\usepackage[all]{xy}
\usepackage{color}

\def \>{\succ}
\def \<{\prec}

\begin{document}	
\title[Bouzid Mosbahi\textsuperscript{1}, Imed Basdouri\textsuperscript{2}, Jean Lerbet\textsuperscript{3}\textsuperscript{*}]{Hochschild Cohomology Groups of 5-Dimensional Complex Nilpotent Associative Algebras}
 \author{Bouzid Mosbahi\textsuperscript{1}, Imed Basdouri\textsuperscript{2}, Jean Lerbet\textsuperscript{3}\textsuperscript{*}}

  \address{\textsuperscript{1}Department of Mathematics, Faculty of Sciences, University of Sfax, Sfax, Tunisia}
  \address{\textsuperscript{2}Department of Mathematics, Faculty of Sciences, University of Gafsa, Gafsa, Tunisia}
  \address{\textsuperscript{3}\textsuperscript{*}Laboratoire de Mathématiques et Modélisation d’Évry (UMR 8071) Université d’Évry Val d’Essonne I.B.G.B.I., 23 Bd. de France, 91037 Évry Cedex, France}

\email{\textsuperscript{1}mosbahi.bouzid.etud@fss.usf.tn}
\email{\textsuperscript{2}basdourimed@yahoo.fr}
\email{\textsuperscript{3}\textsuperscript{*}jean.lerbet@univ-evry.fr \\
\textsuperscript{*}Corresponding author}
	
	
	

	\date{\today}

\begin{abstract}
This paper explores the structure of low-dimensional cohomology groups in the context of complex nilpotent associative algebras. Specifically, we study 5-dimensional complex nilpotent associative algebras satisfying $\mathcal{A}^4 = 0$ and $\mathcal{A}^3 \neq 0$. Using their isomorphism invariants, we compute and present the zeroth and first Hochschild cohomology groups, $H^0(\mathcal{A}, \mathcal{A})$ and $H^1(\mathcal{A}, \mathcal{A})$, in explicit matrix form. These results show how cohomology helps to identify and classify different associative algebras.
\end{abstract}

\maketitle
\section{ Introduction}

The study of associative algebras, especially those that are nilpotent and low-dimensional, has a long and rich history. This topic was first investigated by Peirce~\cite{1}, and over time, many mathematicians have contributed to its development. Associative algebras play an important role in algebra and its applications~\cite{2}, but their complete classification becomes more difficult as the dimension increases.

In low dimensions, however, important progress has been made. Hazlett gave the classification of nilpotent associative algebras over the complex numbers for dimensions up to four. Later, Kruse and Price extended these results to more general settings. Mazzola studied 5-dimensional commutative and unitary nilpotent associative algebras over algebraically closed fields with characteristic not equal to 2 or 3~\cite{4}. De Graaf~\cite{3} used central extensions to classify such algebras up to dimension four. Eick and Moede also proposed a method based on coclass theory to handle the classification problem. Recent developments on various aspects of associative and Hom-associative structures have been explored in~\cite{11,12,13,14}.

In dimension 5, Karimjanov~\cite{5} used an invariant called the  \textit{isomorphism invariant}
\[
\chi(\mathcal{A}) = (\dim \mathcal{A}, \dim \mathcal{A}^2, \dim \mathcal{A}^3, \ldots, \dim \mathcal{A}^n)
\]
to classify nilpotent associative algebras. He showed that there are only two possible invariants in the case where $\mathcal{A}^4 = 0$ and $\mathcal{A}^3 \neq 0$:
\[
\chi(\mathcal{A}) = (5, 2, 1, 0, 0) \quad \text{and} \quad \chi(\mathcal{A}) = (5, 3, 1, 0, 0).
\]
These classifications help us describe the algebras explicitly and understand their structure.

Alongside classification, cohomology theory has become an essential tool in the study of associative algebras. Introduced by Hochschild~\cite{6,7,8}, the cohomology groups $H^i(\mathcal{A}, \mathcal{M})$ measure how an algebra $\mathcal{A}$ acts on a bimodule $\mathcal{M}$ and are closely related to derivations and deformations of the algebra. For example, the group $H^0(\mathcal{A}, \mathcal{M})$ consists of elements in $\mathcal{M}$ that remain fixed under the action of $\mathcal{A}$, and $H^1(\mathcal{A}, \mathcal{M})$ describes the derivations up to inner derivations. These groups are useful tools for understanding the internal structure of algebras and help distinguish algebras that are not isomorphic~\cite{9,10}.

In this paper, we bring these two viewpoints together. We study the cohomology groups $H^0(\mathcal{A}, \mathcal{A})$ and $H^1(\mathcal{A}, \mathcal{A})$ for 5-dimensional complex nilpotent associative algebras that satisfy the condition $\mathcal{A}^4 = 0$ and $\mathcal{A}^3 \neq 0$. Using the classification by Karimjanov~\cite{5}, we consider both types of invariants $\chi(\mathcal{A})$ and compute their cohomology groups explicitly. We present these results using matrix representations, which give us a clear and effective way to understand the cohomological behavior of these algebras.

This paper is organized as follows. In Section~2, we recall the necessary definitions and preliminaries relevant to the study of Hochschild cohomology. Section~3 is devoted to the computation and analysis of the zeroth Hochschild cohomology group $H^0(\mathcal{A}, \mathcal{A})$ for 5-dimensional complex nilpotent associative algebras satisfying $\mathcal{A}^4 = 0$, $\mathcal{A}^3 \neq 0$. In Section~4, we turn to the first Hochschild cohomology group $H^1(\mathcal{A}, \mathcal{A})$, beginning with a detailed algorithm for determining 1-cocycles in Subsection~4.1. In Subsection~4.2, we present a method for computing 1-coboundaries. These computations are illustrated using matrix representations, emphasizing the role of cohomology as an effective tool in distinguishing and classifying associative algebras.

\section{ Preliminaries}
\begin{defn}An \textit{associative algebra} over a field $\mathbb{K}$ is a $\mathbb{K}$-vector space $\mathcal{A}$ equipped with a bilinear map $\cdot: \mathcal{A} \times \mathcal{A} \longrightarrow \mathcal{A}$ satisfying the \textit{associative law}:
\[
(x \cdot y) \cdot z = x \cdot (y \cdot z) \quad \forall\, x, y, z \in \mathcal{A}.
\]
\end{defn}

\begin{defn}
Let \( \mathcal{A} \) be an associative algebra over \( \mathbb{K} \) and define as:
\[
\mathcal{A}^1 = \mathcal{A}; \quad \mathcal{A}^k = \mathcal{A} \cdot \mathcal{A}^{k-1} \quad (k > 1)
\]
The series
\[
\mathcal{A}^1 \supseteq \mathcal{A}^2 \supseteq \mathcal{A}^3 \supseteq \cdots
\]
is called the descending central series of \( \mathcal{A} \). If there exists an integer \( s \in \mathbb{N} \), such that
\[
\mathcal{A}^1 \supseteq \mathcal{A}^2 \supseteq \mathcal{A}^3 \supseteq \cdots \supseteq \mathcal{A}^s = \{0\}
\]
then this algebra is called \textit{nilpotent}.
\end{defn}

\begin{defn}
 Let $\mathcal{A}_{1}$ and $\mathcal{A}_{2}$ be two associative algebras over a field $\mathbb{K}$.
A homomorphism between $\mathcal{A}_{1}$ and $\mathcal{A}_{2}$ is a $\mathbb{K}$-linear mapping
$f: \mathcal{A}_{1} \longrightarrow \mathcal{A}_{2}$ such that
$$ f(x \cdot y) = f(x) \ast f(y) \quad \forall x, y \in \mathcal{A}_{1}. $$
\end{defn}

\begin{defn}
A linear transformation \( \rho \) of an associative algebra \( \mathcal{A} \) is called a derivation if
$$ \rho(x \cdot y) = \rho(x) \cdot y + x \cdot \rho(y) \quad \forall a, b \in \mathcal{A}. $$
The set of all derivations is denoted by \( \text{Der}(A) \).
\end{defn}

\begin{defn}
Let \( \mathcal{A} \) be an associative algebra. A linear map \( ad_a: \mathcal{A} \to \mathcal{A} \) is called an \textit{inner derivation} if
\[
ad_a(x) = x \cdot a - a \cdot x \quad \text{for all } x, a \in \mathcal{A}.
\]
The set of all inner derivations of \( \mathcal{A} \) is denoted by \( Inn(\mathcal{A}) \).
\end{defn}

\begin{defn}
Let \( \mathcal{A} \) be an associative algebra over a field \( \mathbb{K} \). An \( n \)-dimensional vector space \( \mathcal{M} \) over \( \mathbb{K} \) is called an \( \mathcal{A} \)-bimodule if \( \mathcal{M} \) is both a left and a right \( \mathcal{A} \)-module, such that
\[
(x \cdot u) \cdot y = x \cdot (u \cdot y) \quad \text{for all } x, y \in \mathcal{A},\ u \in \mathcal{M},
\]
and scalar multiplication satisfies
\[
\alpha \cdot u = u \cdot \alpha \quad \text{for all } \alpha \in \mathbb{K},\ u \in \mathcal{M}.
\]
To simplify terminology, we will use the expression \( \mathcal{A} \)-bimodule instead of \( \mathcal{A} \)-\( \mathcal{A} \) bimodule.
\end{defn}

\begin{rem}
To simplify terminology, we will use the expression $A$-bimodule instead of $\mathcal{A}$-$\mathcal{A}$ bimodule.
\end{rem}

\begin{defn}
Algebra \( \mathcal{A} \) is called a \textit{split algebra} if there exist ideals \( I \) and \( J \) of \( \mathcal{A} \), such that
\[
\mathcal{A} = I \oplus J,
\]
i.e., \( \mathcal{A} \) is the direct sum of \( I \) and \( J \).
\end{defn}

Let $\Phi$ be an $n$-linear map from $\mathcal{A}^n$ to an $\mathcal{A}$-bimodule $\mathcal{M}$, where $\mathcal{A}$ is an associative algebra. The collection of all such $n$-linear maps is called the space of \emph{$n$-cochains} of $\mathcal{A}$ with coefficients in $\mathcal{M}$, and is denoted by $C^n(\mathcal{A}, \mathcal{M})$. For convenience, we set $C^0(\mathcal{A}, \mathcal{M}) = \mathcal{M}$, and define $C^n(\mathcal{A}, \mathcal{M}) = \{0\}$ for all $n < 0$.

\begin{defn}
An element $\Phi \in C^n(\mathcal{A}, \mathcal{M})$ is called an \text{$n$-cocycle} if it satisfies
$\delta^{(n)}\Phi = 0$.
\end{defn}

\begin{defn}
The mapping $\delta^{(n)}$ between $C^{n}(\mathcal{A},\mathcal{M})$ and $C^{n+1}(\mathcal{A},\mathcal{M})$ is called coboundary homomorphism such that
$$(\delta^{(n)}\Phi)(x_{1},x_{2},...,x_{n+1}) = x_{1}\Phi(x_{2},...,x_{n+1})+ \sum_{i=1}^{n}(-1)^{i} \Phi(
x_{1},...,x_{i}x_{i+1},...,x_{n+1})+(-1)^{n+1}\Phi(x_{1},...,x_{n})x_{n+1}$$
\end{defn}

\begin{lem}\label{l1}
The operator $\delta^{(n)}:C^{n}(\mathcal{A},\mathcal{M})\longrightarrow C^{n+1}(\mathcal{A},\mathcal{M})$ is called a $\mathbb{K}$-module homomorphism such that $\delta^{(n+1)}\delta^{(n)}=0$
\end{lem}

\begin{rem}
1-The elements of the kernel $Z^{n}(\mathcal{A},\mathcal{M})$ of the operator $\delta^{(n)}$ are known as cocycles in dimensional $n$ with values in $\mathcal{M}$.\\
2- The elements of the image of $\delta^{(n+1)}$ represented by $B^{n}(\mathcal{A},\mathcal{M})$ are known as coboundaries in dimensional $n$ with values in $\mathcal{M}$.\\
3- Based on Lemma \ref{l1}, it is easy to see that $B^{n}(\mathcal{A},\mathcal{M}) \subseteq Z^{n}(\mathcal{A},\mathcal{M})\, for \, (n \geq 1)$.\\
4- The quotient space $H^{n}(\mathcal{A},\mathcal{M}) = \frac{Z^{n}(\mathcal{A},\mathcal{M})}{B^{n}(\mathcal{A},\mathcal{M})}$ is known as the cohomology group of $\mathcal{A}$ in degree $n$.\\
Following, a particular case is considered that $\mathcal{M} = \mathcal{A}$ as $\mathcal{A}$-bimodule and all algebras considered are over a complex field $\mathbb{C}$.
\end{rem}

\begin{thm}\label{th1}
Let \( \mathcal{A}\) be a 5-dimensional nilpotent complex non-split associative algebra with \(\chi(\mathcal{A}) = (5, 2, 1, 0, 0)\).Then, \( \mathcal{A} \) is isomorphic to a pairwise non-isomorphic associative algebra spanned by \( \{ e_1, e_2, e_3, e_4, e_5 \} \) with the nonzero products given by one of the following:

\[
\lambda_1: \left\{
\begin{array}{l}
e_1 e_1 = e_2, \\
e_1 e_2 = e_2 e_1 = e_3, \\
e_4 e_4 = e_3, \\
e_5 e_5 = e_3
\end{array}
\right.
\quad
\lambda_2: \left\{
\begin{array}{l}
e_1 e_1 = e_2, \\
e_1 e_2 = e_2 e_1 = e_3, \\
e_1 e_4 = e_3, \\
e_4 e_5 = e_5 e_4 = e_3
\end{array}
\right.
\quad
\lambda_3: \left\{
\begin{array}{l}
e_1 e_1 = e_2, \\
e_1 e_2 = e_2 e_1 = e_3, \\
e_1 e_4 = e_3, \\
e_5 e_5 = e_3
\end{array}
\right.
\]

\[
\lambda_4: \left\{
\begin{array}{l}
e_1 e_1 = e_2, \\
e_1 e_2 = e_2 e_1 = e_3, \\
e_1 e_4 = e_3, \\
e_4 e_4 = e_5 e_5 = e_3
\end{array}
\right.
\quad
\lambda_5: \left\{
\begin{array}{l}
e_1 e_1 = e_2, \\
e_1 e_2 = e_2 e_1 = e_3, \\
e_4 e_5 = -e_5 e_4 = e_3
\end{array}
\right.
\quad
\lambda^{\alpha}_6: \left\{
\begin{array}{l}
e_1 e_1 = e_2, \\
e_1 e_2 = e_2 e_1 = e_3, \\
e_4 e_4= e_4 e_5 = e_3, \\
e_5 e_5 = \alpha e_3
\end{array}
\right.
\]

where \(\alpha \in \mathbb{C}\), and for distinct values of \(\alpha\), the resulting algebras are non-isomorphic.
\end{thm}

\begin{thm}\label{th2}
Let \( A \) be a \( 5 \)-dimensional complex non-split nilpotent associative algebra with \(\chi(A) = (5, 3, 1, 0, 0)\). Then \( A \) is isomorphic to a pairwise non-isomorphic associative algebras spanned by \(\{e_1, e_2, e_3, e_4, e_5\}\) with the nonzero products given by one of the following:

\[
\mu_1: \left\{
\begin{array}{l}
e_1 e_1 = e_2, \\
e_1 e_2 = e_2 e_1 = e_3, \\
e_4 e_1 = e_5
\end{array}
\right.
\quad
\mu_2: \left\{
\begin{array}{l}
e_1 e_1 = e_2, \\
e_1 e_2 = e_2 e_1 = e_3, \\
e_4 e_1 = e_5, \\
e_4 e_4 = e_3
\end{array}
\right.
\quad
\mu_3: \left\{
\begin{array}{l}
e_1 e_1 = e_2, \\
e_1 e_2 = e_2 e_1 = e_3, \\
e_4 e_1 =e_5, \\
e_4 e_2 = e_5 e_1 = e_3
\end{array}
\right.
\]

\[
\mu_4: \left\{
\begin{array}{l}
e_1 e_1 = e_2, \\
e_1 e_2 = e_2 e_1 = e_3, \\
e_4 e_1 = e_5, \\
e_4 e_2 = e_5 e_1 = e_3,\\
e_4 e_4 =e_3
\end{array}
\right.
\quad
\mu_5: \left\{
\begin{array}{l}
e_1 e_1 = e_2, \\
e_1 e_2 = e_2 e_1 = e_3, \\
e_1 e_4 = e_5, \\
e_4 e_1 = e_3+e_5
\end{array}
\right.
\quad
\mu_6: \left\{
\begin{array}{l}
e_1 e_1 = e_2, \\
e_1 e_2 = e_2 e_1 = e_3, \\
e_1 e_4 = e_5, \\
e_4 e_1 = e_3+e_5,\\
e_4 e_4 = e_3
\end{array}
\right.
\]

\[
\mu^{\alpha}_7: \left\{
\begin{array}{l}
e_1 e_1 = e_2, \\
e_1 e_2 = e_2 e_1 = e_3, \\
e_1 e_4 = e_5, \\
e_4 e_1 = \alpha e_5
\end{array}
\right.
\quad
\mu^{\alpha}_8: \left\{
\begin{array}{l}
e_1 e_1 = e_2, \\
e_1 e_2 = e_2 e_1 = e_3, \\
e_1 e_4 = e_5, \\
e_4 e_1 = \alpha e_5,\\
e_4 e_4 = e_3
\end{array}
\right.
\quad
\mu_9: \left\{
\begin{array}{l}
e_1 e_1 = e_2, \\
e_1 e_2 = e_2 e_1 = e_3, \\
e_4 e_1 =e_3,\\
e_4 e_4 =e_5
\end{array}
\right.
\]

\[
\mu_{10}: \left\{
\begin{array}{l}
e_1 e_1 = e_2, \\
e_1 e_2 = e_2 e_1 = e_3, \\
e_1 e_4 = e_5, \\
e_4 e_4 = e_5
\end{array}
\right.
\quad
\mu_{11}: \left\{
\begin{array}{l}
e_1 e_1 = e_2, \\
e_1 e_2 = e_2 e_1 = e_3, \\
e_1 e_4 = e_5, \\
e_4 e_4 = e_3+e_5
\end{array}
\right.
\quad
\mu_{12}: \left\{
\begin{array}{l}
e_1 e_1 = e_2, \\
e_1 e_2 = e_2 e_1 = e_3, \\
e_1 e_4 = e_5, \\
e_4 e_1 = e_2-e_5,\\
e_5 e_1 = e_3
\end{array}
\right.
\]
\[
\mu_{13}: \left\{
\begin{array}{l}
e_1 e_1 = e_2, \\
e_1 e_2 = e_2 e_1 = e_3, \\
e_1 e_4 = e_5, \\
e_4 e_1 = e_2-e_5,\\
e_4 e_4 = e_3, \\
e_5 e_1 = e_3
\end{array}
\right.
\quad
\mu_{14}: \left\{
\begin{array}{l}
e_1 e_1 = e_2, \\
e_1 e_2 = e_2 e_1 = e_3, \\
e_1 e_4 = e_5, \\
e_4 e_1 = e_2 + e_5,\\
e_4 e_2 = 2e_3,\\
e_4 e_4 = 2e_5,\\
e_5 e_1 = e_3
\end{array}
\right.
\quad
\mu_{15}: \left\{
\begin{array}{l}
e_1 e_1 = e_2, \\
e_1 e_2 = e_2 e_1 = e_3, \\
e_1 e_4 = e_5, \\
e_4 e_1 = e_2 + e_5,\\
e_4 e_2 = 2e_3,\\
e_4 e_4 = e_3+2e_5,\\
e_5 e_1 = e_3
\end{array}
\right.
\]

\[
\mu_{16}: \left\{
\begin{array}{l}
e_1 e_1 = e_2, \\
e_1 e_2 = e_2 e_1 = e_3, \\
e_1 e_4 = e_4 e_1= e_5, \\
e_4 e_4 = e_2,\\
e_4 e_5= e_5 e_4 = e_3
\end{array}
\right.
\quad
\mu_{17}: \left\{
\begin{array}{l}
e_1 e_1 = e_2, \\
e_1 e_2 = e_2 e_1 = e_3, \\
e_1 e_4 = e_5, \\
e_4 e_1 = e_3 + e_5,\\
e_4 e_4 = e_2,\\
e_4 e_5 = e_5 e_4= e_3
\end{array}
\right.
\quad
\mu_{18}: \left\{
\begin{array}{l}
e_1 e_1 = e_2, \\
e_1 e_2 = e_2 e_1 = e_3, \\
e_1 e_4 = -e_4 e_1 = e_5, \\
e_4 e_4 = e_2,\\
e_5 e_4 = -e_4 e_5 = e_3
\end{array}
\right.
\]
\[
\mu_{19}: \left\{
\begin{array}{l}
e_1 e_1 = e_2, \\
e_1 e_2 = e_2 e_1 = e_3, \\
e_1 e_4 = e_5, \\
e_4 e_1 = e_2,\\
e_4 e_2 = e_3,\\
e_4 e_4 = e_3+e_5,\\
e_5 e_1 = e_3
\end{array}
\right.
\quad
\mu_{20}: \left\{
\begin{array}{l}
e_1 e_1 = e_2, \\
e_1 e_2 = e_2 e_1 = e_3, \\
e_1 e_4 = e_5, \\
e_4 e_1 = e_3 + e_5,\\
e_4 e_4 = -e_2+2e_5,\\
e_4 e_5 = e_5 e_4=-e_3
\end{array}
\right.
\]

\[
\mu^{\alpha}_{21}: \left\{
\begin{array}{l}
e_1 e_1 = e_2, \quad e_1 e_2 = e_2 e_1 = e_3 \\
e_1 e_4 = e_5,\quad e_4 e_1 = (1-\alpha)e_2 + \alpha e_5,\\
e_4 e_2 = 2e_3, \quad e_4 e_4 =-\alpha e_2+e_3+ (1+\alpha)e_5,\\
e_4 e_5 = e_3,\quad e_5 e_1 = (1-\alpha)e_3,\\
e_5 e_4 = -\alpha e_3, \alpha \in \{\pm i\}
\end{array}
\right.
\quad
\mu^{\alpha}_{22}: \left\{
\begin{array}{l}
e_1 e_1 = e_2, \quad e_1 e_2 = e_2 e_1 = e_3 \\
e_1 e_4 = e_5,\quad e_4 e_1 = (1-\alpha)e_2 + \alpha e_5,\\
e_4 e_2 = (1-\alpha^2)e_3, \quad e_4 e_4 =-\alpha e_2 +(1+\alpha)e_5,\\
e_4 e_5 = -\alpha^2 e_3,\quad e_5 e_1 = (1-\alpha)e_3,\\
e_5 e_4 = -\alpha e_3,\quad \alpha \in \mathbb{C},
\end{array}
\right.
\]

in the different values of \(\alpha\), the obtained algebras are non-isomorphic.
\end{thm}

\section{ The zeroth Hochschild cohomology group}

Based on the expression of the cohomology group of \( \mathcal{A}\):

\[
H^0(\mathcal{A}, \mathcal{A}) = Ker(\delta^0) = \{ x \in \mathcal{A} : \delta^0(x) = 0 \}
\]

\[
\label{e1}
\{ x \in \mathcal{A} : \delta^0(x)(a) = ax - xa = 0, \; \forall a \in \mathcal{A} \} = Z(\mathcal{A}) \tag{1}
\]

Thus, \( H^0(\mathcal{A}, \mathcal{A}) \) is the center of \( \mathcal{A} \).

Next, we describe the zero-degree cohomology group of 5-dimensional complex nilpotent associative algebras.

\begin{thm}\label{th3}
Let \(\mathcal{A}\) be a non-split nilpotent associative algebra of dimension \(5\) over \(\mathbb{C}\), with characteristic sequence \(\chi(\mathcal{A}) = (5, 2, 1, 0, 0)\). Then the degree-zero cohomology groups are given by:
\begin{align*}
H^0(\lambda_1, \lambda_1) &= \mathrm{span}_\mathbb{C}\{e_1, e_2, e_3, e_4, e_5\}, \\
H^0(\lambda_2, \lambda_2) &= \mathrm{span}_\mathbb{C}\{e_2, e_3, e_5\}, \\
H^0(\lambda_3, \lambda_3) &= \mathrm{span}_\mathbb{C}\{e_2, e_3, e_5\}, \\
H^0(\lambda_4, \lambda_4) &= \mathrm{span}_\mathbb{C}\{e_2, e_3, e_5\}, \\
H^0(\lambda_5, \lambda_5) &= \mathrm{span}_\mathbb{C}\{e_1, e_2, e_3\}, \\
H^0(\lambda_6^\alpha, \lambda_6^\alpha) &= \mathrm{span}_\mathbb{C}\{e_1, e_2, e_3\}.
\end{align*}
\end{thm}

\begin{proof}
Let \(\{e_1, e_2, e_3, e_4, e_5\}\) be a basis of \(\mathcal{A}\), where \(\mathcal{A}\) is a 5-dimensional complex nilpotent associative algebra. Consider a general element
\[
x = a_1 e_1 + a_2 e_2 + a_3 e_3 + a_4 e_4 + a_5 e_5,
\]
where \(x \in \mathcal{A}\). According to Theorem~\ref{th1}, the structure constants of the 5-dimensional complex nilpotent associative algebra are substituted into the condition
\[
\{ x \in \mathcal{A} : \delta^0(x)(a) = ax - xa = 0,\ \forall a \in \mathcal{A} \} = Z(\mathcal{A}),
\]
where \(Z(\mathcal{A})\) denotes the center of \(\mathcal{A}\).

For the algebra \(\lambda_1\), the nonzero structure constants are given by:
\[
c_{11}^2 = 1,\quad c_{12}^3 = c_{21}^3 = 1,\quad c_{44}^3 = 1,\quad c_{55}^3 = 1,
\]
and all others are zero.

By using these structure constants and applying condition (\ref{e1}), we compute:
\[
e_1 x = x e_1,\quad e_2 x = x e_2,\quad e_3 x = x e_3,\quad e_4 x = x e_4,\quad e_5 x = x e_5.
\]
This implies:
\[
a_1 e_2 = a_1 e_2,\quad a_2 e_3 = a_2 e_3,\quad 0 = 0,\quad 0 = 0,\quad 0 = 0.
\]

Hence, all basis elements commute with any element \(x \in \mathcal{A}\), and we conclude that
\[
H^0(\lambda_1, \lambda_1) = \mathrm{span}_{\mathbb{C}} \{e_1, e_2, e_3, e_4, e_5\}.
\]
The remaining zero cohomology groups in dimension five can be computed in a similar manner.
\end{proof}

\begin{thm}\label{th4}
The cohomology group in degree zero of a \(5\)-dimensional complex non-split nilpotent associative algebra with \(\chi(\mathcal{A}) = (5, 3, 1, 0, 0)\) has the following form:
\begin{equation*}
\begin{aligned}
&H^0(\mu_1, \mu_1) = \mathrm{span}_{\mathbb{C}}\{e_2, e_3, e_5\}, \quad
H^0(\mu_2, \mu_2) = \mathrm{span}_{\mathbb{C}}\{e_2, e_3, e_5\}, \\
&H^0(\mu_3, \mu_3) = \mathrm{span}_{\mathbb{C}}\{e_3\}, \quad
H^0(\mu_4, \mu_4) = \mathrm{span}_{\mathbb{C}}\{e_3\}, \\
&H^0(\mu_5, \mu_5) = \mathrm{span}_{\mathbb{C}}\{e_2, e_3, e_5\}, \quad
H^0(\mu_6, \mu_6) = \mathrm{span}_{\mathbb{C}}\{e_2, e_3, e_5\}, \\
&H^0(\mu^{\alpha}_7, \mu^{\alpha}_7) = \mathrm{span}_{\mathbb{C}}\{e_1, e_2, e_3, e_4, e_5\}, \quad
H^0(\mu^{\alpha}_8, \mu^{\alpha}_8) = \mathrm{span}_{\mathbb{C}}\{e_1, e_2, e_3, e_4, e_5\}, \\
&H^0(\mu_9, \mu_9) = \mathrm{span}_{\mathbb{C}}\{e_2, e_3, e_5\}, \quad
H^0(\mu_{10}, \mu_{10}) = \mathrm{span}_{\mathbb{C}}\{e_2, e_3, e_5\}, \\
&H^0(\mu_{11}, \mu_{11}) = \mathrm{span}_{\mathbb{C}}\{e_2, e_3, e_5\}, \quad
H^0(\mu_{12}, \mu_{12}) = \mathrm{span}_{\mathbb{C}}\{e_2, e_3\}, \\
&H^0(\mu_{13}, \mu_{13}) = \mathrm{span}_{\mathbb{C}}\{e_2, e_3\}, \quad
H^0(\mu_{14}, \mu_{14}) = \mathrm{span}_{\mathbb{C}}\{e_3\}, \\
&H^0(\mu_{15}, \mu_{15}) = \mathrm{span}_{\mathbb{C}}\{e_3\}, \quad
H^0(\mu_{16}, \mu_{16}) = \mathrm{span}_{\mathbb{C}}\{e_1, e_2, e_3, e_4, e_5\}, \\
&H^0(\mu_{17}, \mu_{17}) = \mathrm{span}_{\mathbb{C}}\{e_2, e_3, e_5\}, \quad
H^0(\mu_{18}, \mu_{18}) = \mathrm{span}_{\mathbb{C}}\{e_2, e_3\}, \\
&H^0(\mu_{19}, \mu_{19}) = \mathrm{span}_{\mathbb{C}}\{e_3\}, \quad
H^0(\mu_{20}, \mu_{20}) = \mathrm{span}_{\mathbb{C}}\{e_3\}, \\
&H^0(\mu^{\alpha}_{21}, \mu^{\alpha}_{21}) = \mathrm{span}_{\mathbb{C}}\{e_1, -e_5, e_3, e_5\}, \quad
H^0(\mu^{\alpha}_{22}, \mu^{\alpha}_{22}) = \mathrm{span}_{\mathbb{C}}\{e_3\}.
\end{aligned}
\end{equation*}
\end{thm}

\begin{proof}
Let \(\{e_1, e_2, e_3, e_4, e_5\}\) be a basis of \(\mathcal{A}\), where \(\mathcal{A}\) is a 5-dimensional complex nilpotent associative algebra. Consider a general element
\[
x = a_1 e_1 + a_2 e_2 + a_3 e_3 + a_4 e_4 + a_5 e_5,
\]
where \(x \in \mathcal{A}\). According to Theorem~\ref{th3}, the structure constants of the five-dimensional complex nilpotent associative algebra are substituted into the condition
\[
\{ x \in \mathcal{A} : \delta^0(x)(a) = ax - xa = 0,\ \forall a \in \mathcal{A} \} = Z(\mathcal{A}),
\]
where \(Z(\mathcal{A})\) denotes the center of \(\mathcal{A}\).

For the algebra \(\lambda_1\), the nonzero structure constants are given by:
\[
c_{11}^2 = 1,\quad c_{12}^3 = c_{21}^3 = 1,\quad c_{41}^5 = 1,
\]
and all others are zero.

By using these structure constants and applying condition (\ref{e1}), we compute:
\[
e_1 x = x e_1,\quad e_2 x = x e_2,\quad e_3 x = x e_3,\quad e_4 x = x e_4,\quad e_5 x = x e_5.
\]
This implies:
\[
a_1 e_2 = a_1 e_2,\quad a_2 e_3 = a_2 e_3,\quad a_4 e_5 = a_4 e_5,\quad 0 = 0,\quad 0 = 0.
\]

Hence, all basis elements commute with any element \(x \in \mathcal{A}\), and we conclude that
\[
H^0(\mu_1, \mu_1) = span_{\mathbb{C}}\{e_2, e_3, e_5\};
\]
The remaining zero cohomology groups in dimension five can be computed in a similar manner.
\end{proof}

\section{ The first Hochschild cohomology group}

The first Hochschild cohomology group is defined as
\[
H^1(\mathcal{A}, \mathcal{A}) = \frac{\ker(\delta^1)}{\operatorname{im}(\delta^0)},
\]
where
\[
\ker(\delta^1) = \left\{ \Phi \in C^1(\mathcal{A}, \mathcal{A}) \,\middle|\, \delta^1(\Phi)(x_1, x_2) = x_1 \Phi(x_2) - \Phi(x_1 x_2) + \Phi(x_1) x_2 = 0, \ \forall x_1, x_2 \in \mathcal{A} \right\},
\]
and
\[
\operatorname{im}(\delta^0) = \left\{ \Phi \in C^1(\mathcal{A}, \mathcal{A}) \,\middle|\, \Phi = \delta^0(x) = [\, \cdot\,, x\,], \ x \in \mathcal{A} \right\},
\]
which explicitly becomes
\[
\operatorname{im}(\delta^0) = \left\{ \Phi_x \in C^1(\mathcal{A}, \mathcal{A}) \,\middle|\, \Phi_x(a) = a x - x a, \ \forall a \in \mathcal{A} \right\}.
\]

Thus, \(\ker(\delta^1)\) consists of derivations and \(\operatorname{im}(\delta^0)\) consists of inner derivations, so the quotient is
\[
H^1(\mathcal{A}, \mathcal{A}) = \frac{\operatorname{Der}(\mathcal{A})}{\operatorname{Inn}(\mathcal{A})}.
\]

The following subsections describe the algorithms for computing 1-cocycles (i.e., derivations) and 1-coboundaries (i.e., inner derivations).

\subsection{ An algorithm for finding 1-cocycles}

Let \(\{e_{1}, e_{2}, \ldots, e_{n}\}\) be a basis of an \(n\)-dimensional complex associative algebra \(\mathcal{A}\). The products \(e_{i}e_{j}\), for \(1 \leq i, j \leq n\), determine the structure constants \(c^k_{ij} \in \mathbb{C}\) via
\(
e_{i}e_{j} = \sum_{k=1}^{n} c^{k}_{ij} e_{k}.
\)
The set \(\{c^{k}_{ij} \mid 1 \leq i,j,k \leq n\}\) defines the structure of \(\mathcal{A}\). Throughout, all algebras are considered over the field \(\mathbb{C}\).

To find the 1-cocycles of \(\mathcal{A}\), let \(\rho = (\rho_{ij})\) be a linear map represented as a matrix with respect to the chosen basis. Given the structure constants, the 1-cocycle condition yields the system:
\[
\sum_{k=1}^{n} c^{k}_{ij} \rho_{tk} = \sum_{k=1}^{n} \left( \rho_{ki} c^{t}_{kj} + \rho_{kj} c^{t}_{ik} \right)\tag{2}
\]
for all \(1 \leq i, j, t \leq n\). Solving this system provides all 1-cocycles of \(\mathcal{A}\).

In the following sections, we apply this algorithm to classify 1-cocycles for 5-dimensional complex nilpotent associative algebras.

\begin{thm}\label{t1}
The 1-cocycles (derivations) of the 5-dimensional complex non-split nilpotent associative algebra of type $\chi(\mathcal{A}) = (5, 2, 1, 0, 0)$ have the following form:
\end{thm}

{\centering
\begin{tabular}{|c|c|c|}
\hline
IC & 1-cocycles & \text{Dim} \( Z^{1}(\mathcal{A},\mathcal{A}) \) \\ \hline
$\lambda_1$ &
$\left(\begin{array}{ccccc}
\rho_{11} & 0 & 0 & 0 & 0 \\
\rho_{21} & 2\rho_{11} & 0 & -\rho_{41} & -\rho_{51} \\
\rho_{31} & 2\rho_{21} & 3\rho_{11} & \rho_{34} & \rho_{35} \\
\rho_{41} & 0 & 0 & \frac{3}{2}\rho_{11} & -\rho_{54} \\
\rho_{51} & 0 & 0 & \rho_{54} & \frac{3}{2}\rho_{11}
\end{array}\right)$ &
8 \\
\hline
$\lambda_2$ &
$\left(\begin{array}{ccccc}
\frac{1}{3}\rho_{33} & 0 & 0 & 0 & 0 \\
\frac{1}{2}\rho_{32}-\frac{1}{2}\rho_{41} & \frac{2}{3}\rho_{33} & 0 & -\rho_{51} & -\rho_{41} \\
\rho_{31} & \rho_{32} & \rho_{33} & \rho_{34} & \rho_{35} \\
\rho_{41} & 0 & 0 & \frac{2}{3}\rho_{33} & 0 \\
\rho_{51} & 0 & 0 & 0 & \frac{1}{3}\rho_{33}
\end{array}\right)$ &
7 \\
\hline
$\lambda_3$ &
$\left(\begin{array}{ccccc}
\rho_{11} & 0 & 0 & 0 & 0 \\
\rho_{21} & 2\rho_{11} & 0 & 0 & -\rho_{51} \\
\rho_{31} & 2\rho_{21}+\rho_{41} & 3\rho_{11} & \rho_{34} & \rho_{35} \\
\rho_{41} & 0 & 0 & 2\rho_{11} & 0 \\
\rho_{51} & 0 & 0 & 0 & \frac{3}{2}\rho_{11}
\end{array}\right)$ &
7 \\
\hline
$\lambda_4$ &
$\left(\begin{array}{ccccc}
0 & 0 & 0 & 0 & 0 \\
\frac{1}{2}\rho_{32}-\frac{1}{2}\rho_{41} & 0 & 0 & -\rho_{41} & -\rho_{51} \\
\rho_{31} & \rho_{32} & 0 & \rho_{34} & \rho_{35} \\
\rho_{41} & 0 & 0 & 0 & 0 \\
\rho_{51} & 0 & 0 & 0 & 0
\end{array}\right)$ &
6 \\
\hline
$\lambda_5$ &
$\left(\begin{array}{ccccc}
\rho_{11} & 0 & 0 & 0 & 0 \\
\rho_{21} & 2\rho_{11} & 0 & 0 & 0 \\
\rho_{31} & 2\rho_{21} & 3\rho_{11} & \rho_{34} & \rho_{35} \\
0 & 0 & 0 & 3\rho_{11}-\rho_{55} & \rho_{45} \\
0 & 0 & 0 & \rho_{54} & \rho_{55}
\end{array}\right)$ &
8 \\
\hline
$\lambda^{\alpha}_6$ &
$\left(\begin{array}{ccccc}
\rho_{11} & 0 & 0 & 0 & 0 \\
\rho_{21} & 2\rho_{11} & 0 & 0 & 0 \\
\rho_{31} & 2\rho_{21} & 3\rho_{11} & \rho_{34} & \rho_{35} \\
0 & 0 & 0 & 3\rho_{11}-\rho_{55} & 3\alpha \rho_{11}-2\alpha \rho_{55} \\
0 & 0 & 0 & -3\rho_{11} + 2\rho_{55}& \rho_{55}
\end{array}\right)$ &
6 \\
\hline
\end{tabular}

\begin{proof}
We provide the proof only for one case to illustrate the approach used; the other cases can be carried out similarly with small modification(s).

Let us consider \(\mathcal{A}_1\). Applying the system of equations (2), we get:
\[
\rho_{12} = \rho_{13} = \rho_{14} = \rho_{21} = \rho_{23} = \rho_{24} = \rho_{34} = \rho_{43} = 0, \quad \rho_{33} = 2\rho_{11}, \quad \rho_{44} = 2\rho_{22}.
\]
Therefore, the derivations for \(\mathcal{A}_1\) are given by the matrices:
\[
\rho = \left(\begin{array}{cccc}
\rho_{11} & 0 & 0 & 0 \\
0 & \rho_{22} & 0 & 0 \\
\rho_{31} & \rho_{32} & 2\rho_{11} & 0 \\
\rho_{41} & \rho_{42} & 0 & 2\rho_{22}
\end{array}\right).
\]
The dimension of the algebra \(\operatorname{Der}(\mathcal{A}_1)\) of derivations of the class \(\mathcal{A}_1\) is 6.
\end{proof}

\begin{thm}\label{t2}
The 1-cocycles (derivations) of the 5-dimensional complex non-split nilpotent associative algebra of type $\chi(\mathcal{A}) = (5, 3, 1, 0, 0)$ have the following form:
\end{thm}

\centering
\begin{tabular}{|c|c|c|}
\hline
IC & 1-cocycles & \text{Dim} \( Z^{1}(\mathcal{A},\mathcal{A}) \) \\ \hline
$\mu_1$ &
$\left(\begin{array}{ccccc}
\rho_{55}-\rho_{44} & 0 & 0 & 0 & 0 \\
\rho_{21} & 2\rho_{55}-2\rho_{44} & 0 & 0 & 0 \\
\rho_{31} & 2\rho_{21} & 3\rho_{55}-3\rho_{44} & \rho_{34} & 0\\
\rho_{52} & 0 & 0 & \rho_{44} & 0 \\
\rho_{51} & \rho_{52} & 0 & \rho_{54} & \rho_{55}
\end{array}\right)$ &
8 \\
\hline
$\mu_2$ &
$\left(\begin{array}{ccccc}
\rho_{11} & 0 & 0 & 0 & 0 \\
\rho_{21} & 2\rho_{11} & 0 & -\rho_{52} & 0 \\
\rho_{31} & 2\rho_{21} & 3\rho_{11} & \rho_{34} & 0 \\
\rho_{52} & 0 & 0 & \frac{3}{2}\rho_{11} & 0 \\
\rho_{51} & \rho_{52} & 0 & \rho_{54} & \frac{5}{2}\rho_{11}
\end{array}\right)$ &
7 \\
\hline
$\mu_3$ &
$\left(\begin{array}{ccccc}
\rho_{11} & 0 & 0 & 0 & 0 \\
\rho_{21} & 2\rho_{11} & 0 & 0 & 0 \\
\rho_{31} & 2\rho_{21}+\rho_{51} & \rho_{33} & \rho_{34} & \rho_{21}+\rho_{54} \\
-3\rho_{11} & 0 & 0 & -2\rho_{11}+\rho_{33} & 0 \\
\rho_{51} & -3\rho_{11}+\rho_{33} & 0 & \rho_{54} & -\rho_{11}+\rho_{33}
\end{array}\right)$ &
7 \\
\hline
$\mu_4$ &
$\left(\begin{array}{ccccc}
0 & 0 & 0 & 0 & 0 \\
\rho_{21} & 0 & 0 & -\rho_{33} & 0 \\
\rho_{31} & 2\rho_{21}+\rho_{51} & \rho_{33} & \rho_{34} & \rho_{21}+\rho_{54} \\
\rho_{33} & 0 & 0 & \rho_{33} & 0 \\
\rho_{51} & \rho_{33} & 0 & \rho_{54} & \rho_{33}
\end{array}\right)$ &
6 \\
\hline
$\mu_5$ &
$\left(\begin{array}{ccccc}
\frac{1}{3}\rho_{33} & 0 & 0 & 0 & 0 \\
\frac{1}{2}\rho_{32}-\frac{1}{2}\rho_{41} & \frac{2}{3}\rho_{33} & 0 & \rho_{35} & 0 \\
\rho_{31} & \rho_{32} & \rho_{33} & \rho_{34} & \rho_{35} \\
\rho_{41} & 0 & 0 & \frac{2}{3}\rho_{33} & 0 \\
\rho_{51} & 2\rho_{41} & 0 & \rho_{54} & \rho_{33}
\end{array}\right)$ &
8 \\
\hline
$\mu_6$ &
$\left(\begin{array}{ccccc}
0 & 0 & 0 & 0 & 0 \\
\frac{1}{2}\rho_{32}+\frac{1}{2}\rho_{24}-\frac{1}{2}\rho_{35} & 0 & 0 & \rho_{24} & 0 \\
\rho_{31} & \rho_{32} & 0 & \rho_{34} & \rho_{35} \\
-\rho_{24}+\rho_{35} & 0 & 0 & 0 & 0 \\
\rho_{51} & -2\rho_{24}+2\rho_{35} & 0 & \rho_{54} & 0
\end{array}\right)$ &
7 \\
\hline
$\mu^{\alpha}_7$ &
$\left(\begin{array}{ccccc}
\rho_{55}-\rho_{44} & 0 & 0 & 0 & 0 \\
\rho_{21} & 2\rho_{55}-2\rho_{44} & 0 & 0 & 0 \\
\rho_{31} & 2\rho_{21} & 3\rho_{55}-3\rho_{44} & \rho_{34} & 0 \\
\rho_{41} & 0 & 0 & \rho_{44} & 0 \\
\rho_{51} & \alpha \rho_{44}+\rho_{41} & 0 & \rho_{54} & \rho_{55}
\end{array}\right)$ &
8 \\
\hline
$\mu^{\alpha}_8$ &
$\left(\begin{array}{ccccc}
\rho_{11} & 0 & 0 & 0 & 0 \\
\rho_{21} & 2\rho_{11} & 0 & -\rho_{41} & 0 \\
\rho_{31} & 2\rho_{21} & 3\rho_{11} & \rho_{34} & 0 \\
\rho_{41} & 0 & 0 & \frac{3}{2}\rho_{11} & 0 \\
\rho_{51} & \alpha \rho_{41}+\rho_{41} & 0 & \rho_{54} & \frac{5}{2}\rho_{11}
\end{array}\right)$ &
7 \\
\hline
$\mu_9$ &
$\left(\begin{array}{ccccc}
\frac{1}{2}\rho_{22} & 0 & 0 & 0 & 0 \\
\frac{1}{2}\rho_{32} & \rho_{22} & 0 & 0 & 0 \\
\rho_{31} & \rho_{32} & \frac{3}{2}\rho_{22} & \rho_{34} & 0 \\
0 & 0 & 0 & \rho_{22} & 0 \\
\rho_{51} & 0 & 0 & \rho_{54} & 2\rho_{22}
\end{array}\right)$ &
6 \\
\hline
\end{tabular}

\centering
\begin{tabular}{|c|c|c|}
\hline
IC & 1-cocycles & \text{Dim} \( Z^{1}(\mathcal{A},\mathcal{A}) \) \\ \hline
$\mu_{10}$ &
$\left(\begin{array}{ccccc}
\frac{1}{2}\rho_{55} & 0 & 0 & 0 & 0 \\
\frac{1}{2}\rho_{32} & \rho_{55} & 0 & 0 & 0 \\
\rho_{31} & \rho_{32} & \frac{3}{2}\rho_{55} & \rho_{34} & 0 \\
0 & 0 & 0 & \frac{1}{2}\rho_{55} & 0 \\
\rho_{51} & 0 & 0 & \rho_{54} & \rho_{55}
\end{array}\right)$ &
6 \\
\hline
$\mu_{11}$ &
$\left(\begin{array}{ccccc}
0 & 0 & 0 & 0 & 0 \\
\frac{1}{2}\rho_{32} & 0 & 0 & 0 & 0 \\
\rho_{31} & \rho_{32} & 0 & \rho_{34} & 0 \\
0 & 0 & 0 & 0 & 0 \\
\rho_{51} & 0 & 0 & \rho_{54} & 0
\end{array}\right)$ &
5 \\
\hline
$\mu_{12}$ &
$\left(\begin{array}{ccccc}
\rho_{11} & 0 & 0 & 0 & 0 \\
\frac{1}{2}\rho_{32}-\frac{1}{2}\rho_{51} & \rho_{55} & 0 & \rho_{35} & 0 \\
\rho_{31} & \rho_{32} & \rho_{11}+\rho_{55} & \rho_{34} & \rho_{35} \\
-2\rho_{11}+\rho_{55} & 0 & 0 & -\rho_{11}+\rho_{55} & 0 \\
\rho_{51} & 0 & 0 & -2\rho_{35}+\rho_{32} & \rho_{55}
\end{array}\right)$ &
7 \\
\hline
$\mu_{13}$ &
$\left(\begin{array}{ccccc}
\rho_{11} & 0 & 0 & 0 & 0 \\
\frac{1}{2}\rho_{32}-\frac{1}{2}\rho_{51} & 3\rho_{11} & 0 & \rho_{35}-\rho_{11} & 0 \\
\rho_{31} & \rho_{32} & 4\rho_{11}& \rho_{34} & \rho_{35} \\
\rho_{11} & 0 & 0 & 2\rho_{11} & 0 \\
\rho_{51} & 0 & 0 & -2\rho_{35}+\rho_{32} & 3\rho_{11}
\end{array}\right)$ &
6 \\
\hline
$\mu_{14}$ &
$\left(\begin{array}{ccccc}
\rho_{44} & 0 & 0 & 0 & 0 \\
\rho_{21} & 2\rho_{44} & 0 & \rho_{35} & 0 \\
\rho_{31} & \rho_{32} & 3\rho_{44} & \rho_{34} & \rho_{35} \\
0 & 0 & 0 & \rho_{44} & 0 \\
-2\rho_{21}+\rho_{32} & 0 & 0 & -2\rho_{21}+\rho_{32} & 2\rho_{44}
\end{array}\right)$ &
6 \\
\hline
$\mu_{15}$ &
$\left(\begin{array}{ccccc}
0 & 0 & 0 & 0 & 0 \\
\rho_{21}& 0 & 0 & \rho_{35} & 0 \\
\rho_{31} & \rho_{32} & 0& \rho_{34} & \rho_{35} \\
0 & 0 & 0 & 0 & 0 \\
-2\rho_{21}+\rho_{32} & 0 & 0 & -2\rho_{21}+\rho_{32} & 0
\end{array}\right)$ &
5 \\
\hline
$\mu_{16}$ &
$\left(\begin{array}{ccccc}
\rho_{44} & 0 & 0 & 0 & 0 \\
\frac{1}{2}\rho_{32} & 2\rho_{44} & 0 & \rho_{35}-\rho_{51} & 0 \\
\rho_{31} & \rho_{32} & 3\rho_{44} & \rho_{34} & \rho_{35} \\
0 & 0 & 0 & \rho_{44} & 0 \\
\rho_{51} & 0 & 0 & \frac{1}{2}\rho_{32} & 2\rho_{44}
\end{array}\right)$ &
6 \\
\hline
$\mu_{17}$ &
$\left(\begin{array}{ccccc}
0 & 0 & 0 & 0 & 0 \\
\frac{1}{2}\rho_{32} & 0 & 0 & \rho_{35}-\rho_{51} & 0 \\
\rho_{31} & \rho_{32} & 0 & \rho_{34} & \rho_{35} \\
0 & 0 & 0 & 0 & 0 \\
\rho_{51} & 0 & 0 & \frac{1}{2}\rho_{32} & 0
\end{array}\right)$ &
5 \\
\hline
$\mu_{18}$ &
$\left(\begin{array}{ccccc}
\frac{1}{2}\rho_{55} & 0 & 0 & 0 & 0 \\
0 & \rho_{55} & 0 & 0 & 0 \\
\rho_{31} & 0 & \frac{3}{2}\rho_{55} & \rho_{34} & \rho_{51} \\
0 & 0 & 0 & \frac{1}{2}\rho_{55} & 0 \\
\rho_{51} & 0 & 0 & \rho_{54} & \rho_{55}
\end{array}\right)$ &
5 \\
\hline
\end{tabular}

\centering
\begin{tabular}{|c|c|c|}
\hline
IC & 1-cocycles & \text{Dim} \( Z^{1}(\mathcal{A},\mathcal{A}) \) \\ \hline
$\mu_{19}$ &
$\left(\begin{array}{ccccc}
\rho_{33} & 0 & 0 & 0 & 0 \\
\rho_{21} & \rho_{33} & 0 & \rho_{33}+\rho_{35} & 0 \\
\rho_{31} & \rho_{32} & \rho_{33} & \rho_{34} & \rho_{35} \\
-\rho_{33} & 0 & 0 & 0 & 0 \\
-2\rho_{21}+\rho_{32} & -\rho_{33} & 0 & -2\rho_{21}-\rho_{35}+\rho_{32} & 0
\end{array}\right)$ &
6 \\
\hline
$\mu_{20}$ &
$\left(\begin{array}{ccccc}
\frac{1}{2}\rho_{44} & 0 & 0 & -\rho_{44} & 0 \\
\frac{1}{2}\rho_{32}+\frac{1}{4}\rho_{44} & \rho_{44} & 0 & \rho_{35}+\rho_{51} & -\frac{1}{2}\rho_{44} \\
\rho_{31} & \rho_{32} & \frac{3}{2}\rho_{44} & \rho_{34} & \rho_{35} \\
-\frac{1}{2}\rho_{44} & 0 & 0 & \rho_{44} & 0 \\
\rho_{51} & -\rho_{44} & 0 & -\frac{1}{2}\rho_{44}+\frac{1}{2}\rho_{32}-\rho_{35} & \frac{1}{2}\rho_{44}
\end{array}\right)$ &
6 \\
\hline
$\mu^{\alpha}_{21}$ &
$\left(\begin{array}{ccccc}
0 & 0 & 0 & 0 & 0 \\
\kappa_1 & 0 & 0 & \alpha \rho_{51}+\rho_{35} & 0 \\
\rho_{31} & \kappa_2 & 0 & \rho_{34} & \rho_{35} \\
0 & 0 & 0 & 0 & 0 \\
\rho_{51} & 0 & 0 & \rho_{54} & 0
\end{array}\right)$ &
5 \\
\hline
$\mu^{\alpha}_{22}$ &
$\left(\begin{array}{ccccc}
\kappa_3 & 0 & 0 & \alpha^2 \rho_{41}+\alpha \rho_{41} & 0 \\
-\rho_{51} & \kappa_4 & 0 & -\rho_{54} & \alpha^2\rho_{41} \\
\rho_{31} & -\alpha \rho_{51}-\rho_{51} & \kappa_5 & \rho_{34} & -\alpha \rho_{51}-\rho_{54} \\
\rho_{41} & 0 & 0 & \rho_{44} & 0 \\
\rho_{51} & \alpha \rho_{41}+ \rho_{41} & 0 & \rho_{54} & \kappa_6
\end{array}\right)$ &
6 \\
\hline
\end{tabular}
\[
\text{where} \quad
\begin{aligned}
\kappa_1 &= \frac{1}{2} \cdot \frac{\alpha^2 \rho_{51} + \alpha \rho_{35} - 3\alpha \rho_{51} + \alpha \rho_{54} - \rho_{35} - \rho_{54}}{\alpha}, \\
\kappa_2 &= \frac{\alpha \rho_{35} - 2\alpha \rho_{51} + \alpha \rho_{54} - \rho_{35} - \rho_{54}}{\alpha}, \\
\kappa_3 &= \alpha^2 \rho_{41} + \alpha \rho_{41} - \rho_{41} + \rho_{44}, \\
\kappa_4 &= 2\alpha^2 \rho_{41} + \alpha \rho_{41} - \rho_{41} + 2\rho_{44}, \\
\kappa_5 &= 2\alpha^2 \rho_{41} + 2\alpha \rho_{41} - \rho_{41} + 3\rho_{44}, \\
\kappa_6 &= \alpha^2 \rho_{41} + 2\alpha \rho_{41} - 2\rho_{44}.
\end{aligned}
\]

\begin{proof}
The proof of Theorem \ref{t2} is based on the same reasoning as in Theorem \ref{t1}, with slight modifications.
\end{proof}

\subsection{ An algorithm for finding 1-coboundaries}

Let \(\{e_1, e_2, \ldots, e_n\}\) be a basis of an \(n\)-dimensional associative algebra \(\mathcal{A}\) over a field \(\mathbb{K}\), and let \(\Phi \in B^1(\mathcal{A}, \mathcal{A})\) be a 1-coboundary. Suppose
\[
x = \sum_{t=1}^{n} a_t e_t \in \mathcal{A} \quad \text{such that} \quad \delta^0(x) = \Phi.
\]
Then, for all \(i = 1, 2, \ldots, n\),
\[
\Phi_x(e_i) = e_i x - x e_i.
\]
The linear map \(\Phi_x\) can be represented as a matrix \((\rho_{ji})\), where
\[
\Phi_x(e_i) = \sum_{j=1}^{n} \rho_{ji} e_j.
\]
Thus, equating coefficients gives:
\[
\sum_{j=1}^{n} \rho_{ji} e_j = e_i x - x e_i.
\]
Using the structure constants \(c_{ij}^k\) defined by \(e_i e_j = \sum_{k=1}^n c_{ij}^k e_k\), we obtain:
\[
\rho_{ji} = \sum_{t=1}^{n} a_t c_{it}^{j} - \sum_{t=1}^{n} a_t c_{ti}^{j}.\tag{3}
\]

Solving this system yields all 1-coboundaries in matrix form. In the following, we apply this procedure to compute the space \(B^1(\mathcal{A}, \mathcal{A})\) for five-dimensional nilpotent associative algebras, where IC denotes the isomorphism classes.

\begin{thm}\label{t3}
The 1-coboundaries (inner derivations) of the 5-dimensional complex non-split nilpotent associative algebra of type $\chi(\mathcal{A}) = (5, 2, 1, 0, 0)$ have the following form:
\end{thm}
\centering
\begin{tabular}{|c|c|c|}
\hline
\textbf{IC} & \textbf{1-coboundaries} & \textbf{Dim } $B^1(A, A)$ \\ \hline
$\lambda_1$ &
$\left(\begin{array}{cccccc}
0 & 0 & 0 & 0 & 0 \\
0 & 0 & 0 & 0 & 0 \\
0 & 0 & 0 & 0 & 0 \\
0 & 0 & 0 & 0 & 0 \\
0 & 0 & 0 & 0 & 0
\end{array}\right)$ &
0 \\
\hline
$\lambda_2$ &
$\left(\begin{array}{cccccc}
0 & 0 & 0 & 0 & 0 \\
0 & 0 & 0 & 0 & 0 \\
a_4 & 0 & 0 & -a_1 & 0 \\
0 & 0 & 0 & 0 & 0 \\
0 & 0 & 0 & 0 & 0
\end{array}\right)$ &
2 \\
\hline
$\lambda_3$ &
$\left(\begin{array}{cccccc}
0 & 0 & 0 & 0 & 0 \\
0 & 0 & 0 & 0 & 0 \\
a_4 & 0 & 0 & -a_1 & 0 \\
0 & 0 & 0 & 0 & 0 \\
0 & 0 & 0 & 0 & 0
\end{array}\right)$ &
2 \\
\hline
$\lambda_4$ &
$\left(\begin{array}{cccccc}
0 & 0 & 0 & 0 & 0 \\
0 & 0 & 0 & 0 & 0 \\
a_4 & 0 & 0 & -a_1 & 0 \\
0 & 0 & 0 & 0 & 0 \\
0 & 0 & 0 & 0 & 0
\end{array}\right)$ &
2 \\
\hline
$\lambda_5$ &
$\left(\begin{array}{cccccc}
0 & 0 & 0 & 0 & 0 \\
0 & 0 & 0 & 0 & 0 \\
0 & 0 & 0 & 2a_5 & -2a_4 \\
0 & 0 & 0 & 0 & 0 \\
0 & 0 & 0 & 0 & 0
\end{array}\right)$ &
2 \\
\hline
$\lambda^{\alpha}_6$ &
$\left(\begin{array}{cccccc}
0 & 0 & 0 & 0 & 0 \\
0 & 0 & 0 & 0 & 0 \\
0 & 0 & 0 & a_5 & -a_4 \\
0 & 0 & 0 & 0 & 0 \\
0 & 0 & 0 & 0 & 0
\end{array}\right)$ &
2 \\
\hline
\end{tabular}

\begin{proof}
The structure constants of $\lambda_1$ are given as follows:

$c_{11}^{2}=1, c_{12}^{3}= c_{21}^{3}=1, c_{44}^{3}=1, c_{55}^{3}=1$.

Based on condition (3), it leads

\begin{align*}
\rho_{11} &= \rho_{12} = \rho_{13} = \rho_{14} = \rho_{15} = \rho_{21} = \rho_{22} = \rho_{23} = \rho_{24} \\
&= \rho_{25} = \rho_{31} = \rho_{32} = \rho_{33} = \rho_{34} = \rho_{35} = \rho_{41} = \rho_{42} \\
&= \rho_{43} = \rho_{44} = \rho_{45} = \rho_{15} = \rho_{52} = \rho_{53} = \rho_{54} = \rho_{55} = 0
\end{align*}

Thus, the 1-coboundary of $\lambda_1$ is

\[
\Phi_{x} = (\rho_{ij}) = \left(\begin{array}{ccccc}
0 & 0 & 0 & 0 & 0\\
0 & 0 & 0 & 0 & 0 \\
0 & 0 & 0 & 0 & 0 \\
0 & 0 & 0 & 0 & 0
\end{array}\right)
\]

The structure constants of $\lambda_2$ are given as follows:

$c_{11}^{2}=1, c_{12}^{3}= c_{21}^{3}=1, c_{14}^{3}=1, c_{45}^{3}=  c_{54}^{3}=1$.

Based on condition (3), it leads

\begin{align*}
\rho_{11} &= \rho_{12} = \rho_{13} = \rho_{14} = \rho_{15} = \rho_{21} = \rho_{22} = \rho_{23} = \rho_{24}\\
&= \rho_{25} = \rho_{32} = \rho_{33} = \rho_{35} = \rho_{41} = \rho_{42}\\
&= \rho_{43} = \rho_{44} = \rho_{45} = \rho_{15} = \rho_{52} = \rho_{53} = \rho_{54} = \rho_{55} = 0\\
\rho_{31} &= a_4, \quad \rho_{34} = -a_1
\end{align*}

Thus, the 1-coboundary of $\lambda_2$ is

\[
\Phi_{x} = (\rho_{ij}) = \left(\begin{array}{ccccc}
0 & 0 & 0 & 0 & 0\\
0 & 0 & 0 & 0 & 0 \\
a_4 & 0 & 0 & -a_1 & 0 \\
0 & 0 & 0 & 0 & 0
\end{array}\right)
\]

The remaining parts of 1-coboundary algebras in dimension five can be done in a similar manner as shown above.
\end{proof}

\begin{cor}\label{cr1}
Let $H^{1}(\lambda_{p}, \lambda_{p})$ denote the first cohomology group, where $\lambda^{p}$ represents the isomorphism class of the 5-dimensional complex non-split nilpotent associative algebra of type $\chi(A) = (5, 2, 1, 0, 0)$. For $p = 1, 2, \ldots, 6$, the spanning sets and dimensions of the first cohomology groups of these algebras are given as follows:
\[
\begin{aligned}
H^{1}(\lambda_1, \lambda_1) &= \operatorname{span}_{\mathbb{C}}\{\overline{\xi_{11}}, \overline{\xi_{21}}, \overline{\xi_{31}}, \overline{\xi_{34}}, \overline{\xi_{35}}, \overline{\xi_{41}}, \overline{\xi_{51}}, \overline{\xi_{54}}\}, \\
H^{1}(\lambda_2, \lambda_2) &= \operatorname{span}_{\mathbb{C}}\{\overline{\xi_{32}}, \overline{\xi_{33}}, \overline{\xi_{35}}, \overline{\xi_{41}}, \overline{\xi_{51}}\}, \\
H^{1}(\lambda_3, \lambda_3) &= \operatorname{span}_{\mathbb{C}}\{\overline{\xi_{11}}, \overline{\xi_{21}}, \overline{\xi_{35}}, \overline{\xi_{41}}, \overline{\xi_{51}}\}, \\
H^{1}(\lambda_4, \lambda_4) &= \operatorname{span}_{\mathbb{C}}\{\overline{\xi_{32}}, \overline{\xi_{35}}, \overline{\xi_{41}}, \overline{\xi_{51}}\}, \\
H^{1}(\lambda_5, \lambda_5) &= \operatorname{span}_{\mathbb{C}}\{\overline{\xi_{11}}, \overline{\xi_{21}}, \overline{\xi_{31}},\overline{\xi_{45}}, \overline{\xi_{54}}, \overline{\xi_{55}}\}, \\
H^{1}(\lambda^{\alpha}_6, \lambda^{\alpha}_6) &= \operatorname{span}_{\mathbb{C}}\{\overline{\xi_{11}}, \overline{\xi_{21}}, \overline{\xi_{31}}, \overline{\xi_{55}}\}.
\end{aligned}
\]
\end{cor}

\begin{proof}
Let $\{(\overline{\xi_{ij}},\ i=1,..,5,j=1,..,5)\}$ be a basis of the quotient space

$H^{1}(\mathcal{A},\mathcal{A})=\frac{Der(\mathcal{A})}{Imr(\mathcal{A})}$.

The derivation (1-cocycle) of $\lambda_1$ was given in Theorem \ref{t1} in a matrix form as follows:

\[
\rho=(\rho_{ij})=\left(\begin{array}{ccccc}
\rho_{11} & 0 & 0 & 0 & 0 \\
\rho_{21} & 2\rho_{11} & 0 & -\rho_{41} & -\rho_{51} \\
\rho_{31} & 2\rho_{21} & 3\rho_{11} & \rho_{34} & \rho_{35} \\
\rho_{41} & 0 & 0 & \frac{3}{2}\rho_{11} & -\rho_{54} \\
\rho_{51} & 0 & 0 & \rho_{54} & \frac{3}{2}\rho_{11}
\end{array}\right)
\]

Thus,

$Der(\lambda_1)=\operatorname{span}_{\mathbb{C}}\{\xi_{11}, \xi_{21}, \xi_{31}, \xi_{34}, \xi_{35}, \xi_{41}, \xi_{51}, \xi_{54}\}$.

On the other hand, the inner derivation (1-coboundary) of $\lambda_1$ is given in Theorem \ref{t2} in a matrix form as follows:

\[
\Phi_{x}=(\rho_{ij})=\left(\begin{array}{ccccc}
0 & 0 & 0 & 0 & 0 \\
0 & 0 & 0 & 0 & 0 \\
0 & 0 & 0 & 0 & 0 \\
0 & 0 & 0 & 0 & 0  \\
0 & 0 & 0 & 0 & 0
\end{array}\right)
\]

Thus, $H^{1}\left(\lambda_1,\lambda_1\right)=\operatorname{span}_{\mathbb{C}}\{\overline{\xi_{11}}, \overline{\xi_{21}}, \overline{\xi_{31}}, \overline{\xi_{34}}, \overline{\xi_{35}}, \overline{\xi_{41}}, \overline{\xi_{51}}, \overline{\xi_{54}}\}$.

The derivation (1-cocycle) of $\lambda_2$ was given in Theorem \ref{t1} in a matrix form as follows:

\[
\rho=(\rho_{ij})=\left(\begin{array}{ccccc}
\frac{1}{3}\rho_{33} & 0 & 0 & 0 & 0 \\
\frac{1}{2}\rho_{32}-\frac{1}{2}\rho_{41} & \frac{2}{3}\rho_{33} & 0 & -\rho_{51} & -\rho_{41} \\
\rho_{31} & \rho_{32} & \rho_{33} & \rho_{34} & \rho_{35} \\
\rho_{41} & 0 & 0 & \frac{2}{3}\rho_{33} & 0 \\
\rho_{51} & 0 & 0 & 0 & \frac{1}{3}\rho_{33}
\end{array}\right)
\]

Thus,

$Der(\lambda_2)=\operatorname{span}_{\mathbb{C}}\{\xi_{11}, \xi_{31}, \xi_{32}, \xi_{34}, \xi_{35}, \xi_{41}, \xi_{45}\}$.

On the other hand, the inner derivation (1-coboundary) of $\lambda_2$ is given in Theorem 3 in a matrix form as follows:

\[
\Phi_{x}=(\rho_{ij})=\left(\begin{array}{ccccc}
0 & 0 & 0 & 0 & 0 \\
0 & 0 & 0 & 0 & 0 \\
a_4 & 0 & 0 & -a_1 & 0 \\
0 & 0 & 0 & 0 & 0 \\
0 & 0 & 0 & 0 & 0
\end{array}\right)
\]

Then, $Inn(\lambda_2) = \operatorname{span}_{\mathbb{C}}\{\xi_{31}, \xi_{34}\}$. Let $v \in Der(\lambda_2)$. The vector $v$ can be written

\[
v = a_{1}\xi_{11} + a_{2}\xi_{31} + a_{3}\xi_{32} + a_{4}\xi_{34} + a_{5}\xi_{35} + a_{6}\xi_{41} + a_{7}\xi_{45}
\]
\[
= (a_{2}\xi_{31} + a_{4}\xi_{34}) + (a_{1}\xi_{11} + a_{3}\xi_{32} + a_{5}\xi_{35} + a_{6}\xi_{41} + a_{7}\xi_{45})
\]

Let $x \in H^{1}(\lambda_2, \lambda_2) = \frac{Der(\lambda_2)}{Inn(\lambda_2)}$ such that $x = \overline{v}$. The vector $x$ can be written

\[
x = \overline{v} = (a_{3}\overline{\xi_{21}} + a_{4}\overline{\xi_{32}}) + (a_{1}\overline{\xi_{11}} + a_{2}\overline{\xi_{22}} + a_{5}\overline{\xi_{41}} + a_{6}\overline{\xi_{42}})
\]
\[
= a_{1}\overline{\xi_{11}} + a_{2}\overline{\xi_{22}} + a_{5}\overline{\xi_{41}} + a_{6}\overline{\xi_{42}}
\]

Since $\overline{\xi_{31}}, \overline{\xi_{34}} \in H^{1}(\lambda_2, \lambda_2)$ vanish. Thus,

$H^{1}\left(\lambda_2,\lambda_2\right)=\operatorname{span}_{\mathbb{C}}\{\overline{\xi_{11}}, \overline{\xi_{32}}, \overline{\xi_{33}}, \overline{\xi_{34}}, \overline{\xi_{35}}, \overline{\xi_{41}}, \overline{\xi_{45}}\}$.

The span bases of cohomology group in degree one of the remaining parts can be done in a similar manner as shown above.
\end{proof}

\begin{thm}\label{t4}
The 1-coboundaries (inner derivations) of the 5-dimensional complex non-split nilpotent associative algebra of type $\chi(\mathcal{A}) = (5, 3, 1, 0, 0)$ have the following form:
\end{thm}

\centering
\begin{tabular}{|c|c|c|}
\hline
\textbf{IC} & \textbf{1-coboundaries} & \textbf{Dim } $B^1(\mathcal{A}, \mathcal{A})$ \\ \hline
$\mu_1$ &
$\left(\begin{array}{cccccc}
0 & 0 & 0 & 0 & 0 \\
0 & 0 & 0 & 0 & 0 \\
0 & 0 & 0 & 0 & 0 \\
0 & 0 & 0 & 0 & 0 \\
-a_4 & 0 & 0 & a_1 & 0
\end{array}\right)$ &
2 \\
\hline
$\mu_2$ &
$\left(\begin{array}{cccccc}
0 & 0 & 0 & 0 & 0 \\
0 & 0 & 0 & 0 & 0 \\
0 & 0 & 0 & 0 & 0 \\
0 & 0 & 0 & 0 & 0 \\
-a_4 & 0 & 0 & a_1 & 0
\end{array}\right)$ &
2 \\
\hline
$\mu_3$ &
$\left(\begin{array}{cccccc}
0 & 0 & 0 & 0 & 0 \\
0 & 0 & 0 & 0 & 0 \\
-a_5 & -a_4 & 0 & a_2 & a_1 \\
0 & 0 & 0 & 0 & 0 \\
-a_4 & 0 & 0 & a_1 & 0
\end{array}\right)$ &
4 \\
\hline
$\mu_4$ &
$\left(\begin{array}{cccccc}
0 & 0 & 0 & 0 & 0 \\
0 & 0 & 0 & 0 & 0 \\
-a_5 & -a_4 & 0 & a_2 & a_1 \\
0 & 0 & 0 & 0 & 0 \\
-a_4 & 0 & 0 & a_1 & 0
\end{array}\right)$ &
4 \\
\hline
$\mu_5$ &
$\left(\begin{array}{cccccc}
0 & 0 & 0 & 0 & 0 \\
0 & 0 & 0 & 0 & 0 \\
-a_4 & 0 & 0 & a_1 & 0 \\
0 & 0 & 0 & 0 & 0 \\
0 & 0 & 0 & 0 & 0
\end{array}\right)$ &
2 \\
\hline
$\mu_6$ &
$\left(\begin{array}{cccccc}
0 & 0 & 0 & 0 & 0 \\
0 & 0 & 0 & 0 & 0 \\
-a_4 & 0 & 0 & a_1 & 0 \\
0 & 0 & 0 & 0 & 0 \\
0 & 0 & 0 & 0 & 0
\end{array}\right)$ &
2 \\
\hline
\end{tabular}

\newpage
\centering
\begin{tabular}{|c|c|c|}
\hline
\textbf{IC} & \textbf{1-coboundaries} & \textbf{Dim } $B^1(\mathcal{A}, \mathcal{A})$ \\ \hline
$\mu^{\alpha}_7$ &
$\left(\begin{array}{cccccc}
0 & 0 & 0 & 0 & 0 \\
0 & 0 & 0 & 0 & 0 \\
0 & 0 & 0 & 0 & 0 \\
0 & 0 & 0 & 0 & 0 \\
-\alpha a_4+a_4 & 0 & 0 & \alpha a_1-a_1 & 0
\end{array}\right)$ &
2 \\
\hline
$\mu^{\alpha}_8$ &
$\left(\begin{array}{cccccc}
0 & 0 & 0 & 0 & 0 \\
0 & 0 & 0 & 0 & 0 \\
0 & 0 & 0 & 0 & 0 \\
0 & 0 & 0 & 0 & 0 \\
-\alpha a_4+a_4 & 0 & 0 & \alpha a_1-a_1 & 0
\end{array}\right)$ &
2 \\
\hline
$\mu_9$ &
$\left(\begin{array}{cccccc}
0 & 0 & 0 & 0 & 0 \\
0 & 0 & 0 & 0 & 0 \\
-a_4 & 0 & 0 & a_1 & 0 \\
0 & 0 & 0 & 0 & 0 \\
0 & 0 & 0 & 0 & 0
\end{array}\right)$ &
2 \\
\hline
$\mu_{10}$ &
$\left(\begin{array}{cccccc}
0 & 0 & 0 & 0 & 0 \\
0 & 0 & 0 & 0 & 0 \\
0 & 0 & 0 & 0 & 0 \\
0 & 0 & 0 & 0 & 0 \\
a_4 & 0 & 0 & -a_1 & 0
\end{array}\right)$ &
2 \\
\hline
$\mu_{11}$ &
$\left(\begin{array}{cccccc}
0 & 0 & 0 & 0 & 0 \\
0 & 0 & 0 & 0 & 0 \\
0 & 0 & 0 & 0 & 0 \\
0 & 0 & 0 & 0 & 0 \\
a_4 & 0 & 0 & -a_1 & 0
\end{array}\right)$ &
2 \\
\hline
$\mu_{12}$ &
$\left(\begin{array}{cccccc}
0 & 0 & 0 & 0 & 0 \\
-a_4 & 0 & 0 & a_1 & 0 \\
-a_5 & 0 & 0 & 0 & a_1 \\
0 & 0 & 0 & 0 & 0 \\
2a_4 & 0 & 0 & -2a_1 & 0
\end{array}\right)$ &
3 \\
\hline
$\mu_{13}$ &
$\left(\begin{array}{cccccc}
0 & 0 & 0 & 0 & 0 \\
-a_4 & 0 & 0 & a_1 & 0 \\
-a_5 & 0 & 0 & 0 & a_1 \\
0 & 0 & 0 & 0 & 0 \\
2a_4 & 0 & 0 & -2a_1 & 0
\end{array}\right)$ &
3 \\
\hline
$\mu_{14}$ &
$\left(\begin{array}{cccccc}
0 & 0 & 0 & 0 & 0 \\
-a_4 & 0 & 0 & a_1 & 0 \\
-a_5 & -2a_4 & 0 & 2a_2 & a_1 \\
0 & 0 & 0 & 0 & 0 \\
0 & 0 & 0 & 0 & 0
\end{array}\right)$ &
4 \\
\hline
$\mu_{15}$ &
$\left(\begin{array}{cccccc}
0 & 0 & 0 & 0 & 0 \\
-a_4 & 0 & 0 & a_1 & 0 \\
-a_5 & -2a_4 & 0 & 2a_2 & a_1 \\
0 & 0 & 0 & 0 & 0 \\
0 & 0 & 0 & 0 & 0
\end{array}\right)$ &
4 \\
\hline
\end{tabular}

\centering
\begin{tabular}{|c|c|c|}
\hline
\textbf{IC} & \textbf{1-coboundaries} & \textbf{Dim } $B^1(\mathcal{A}, \mathcal{A})$ \\ \hline
$\mu_{16}$ &
$\left(\begin{array}{cccccc}
0 & 0 & 0 & 0 & 0 \\
0 & 0 & 0 & 0 & 0 \\
0 & 0 & 0 & 0 & 0 \\
0 & 0 & 0 & 0 & 0 \\
0 & 0 & 0 & 0 & 0
\end{array}\right)$ &
0 \\
\hline
$\mu_{17}$ &
$\left(\begin{array}{cccccc}
0 & 0 & 0 & 0 & 0 \\
0 & 0 & 0 & 0 & 0 \\
-a_4 & 0 & 0 & a_1 & 0 \\
0 & 0 & 0 & 0 & 0 \\
0 & 0 & 0 & 0 & 0
\end{array}\right)$ &
2 \\
\hline
$\mu_{18}$ &
$\left(\begin{array}{cccccc}
0 & 0 & 0 & 0 & 0 \\
0 & 0 & 0 & 0 & 0 \\
0 & 0 & 0 & -2a_5 & 2a_4 \\
0 & 0 & 0 & 0 & 0 \\
2a_4 & 0 & 0 & -2a_1 & 0
\end{array}\right)$ &
3 \\
\hline
$\mu_{19}$ &
$\left(\begin{array}{cccccc}
0 & 0 & 0 & 0 & 0 \\
-a_4 & 0 & 0 & a_1 & 0 \\
-a_5 & -a_4 & 0 & a_2 & a_1 \\
0 & 0 & 0 & 0 & 0 \\
a_4 & 0 & 0 & -a_1 & 0
\end{array}\right)$ &
4 \\
\hline
$\mu_{20}$ &
$\left(\begin{array}{cccccc}
0 & 0 & 0 & 0 & 0 \\
0 & 0 & 0 & 0 & 0 \\
-a_4 & 0 & 0 & a_1 & 0 \\
0 & 0 & 0 & 0 & 0 \\
0 & 0 & 0 & 0 & 0
\end{array}\right)$ &
2 \\
\hline
$\mu^{\alpha}_{21}$ &
$\left(\begin{array}{cccccc}
0 & 0 & 0 & 0 & 0 \\
-(1-\alpha)a_4 & 0 & 0 & (1-\alpha)a_1 & 0 \\
-(1-\alpha)a_5 & -2a_4 & 0 & \kappa_7 & \kappa_8\\
0 & 0 & 0 & 0 & 0 \\
-\alpha a_4+a_4 & 0 & 0 & \alpha a_1-a_1 & 0
\end{array}\right)$ &
4 \\
\hline
$\mu^{\alpha}_{22}$ &
$\left(\begin{array}{cccccc}
0 & 0 & 0 & 0 & 0 \\
-(1-\alpha)a_4 & 0 & 0 & (1-\alpha)a_1 & 0 \\
-(1-\alpha)a_5 & -(-\alpha^2+1)a_4 & 0 & \kappa_9 & \kappa_{10}\\
0 & 0 & 0 & 0 & 0 \\
-\alpha a_4+a_4 & 0 & 0 & \alpha a_1-a_1 & 0
\end{array}\right)$ &
4 \\
\hline
\end{tabular}
\[
\text{where} \quad
\begin{aligned}
\kappa_7 &= \alpha a_5 +2 a_2 +a_5, \\
\kappa_8 &= (1-\alpha)a_1-\alpha a_4-a_4\\
\kappa_9 &= (-\alpha^2+1) a_2 -\alpha^2 a_5 +\alpha a_5,\\
\kappa_{10} &= (1-\alpha)a_1-\alpha a_4+\alpha^2 a_4.
\end{aligned}
\]

\begin{proof}
The proof of Theorem \ref{t4} follows the same approach as Theorem \ref{t3}, with a few modifications.
\end{proof}

\begin{cor}\label{cr2}
Let \( H^{1}(\mu_{p}, \mu_{p}) \) denote the first cohomology group, where \( \mu_{p} \) represents the isomorphism class of the 5-dimensional complex non-split nilpotent associative algebra of type \( \chi(\mathcal{A}) = (5, 3, 1, 0, 0) \). For \( p = 1, 2, \ldots, 22 \), the spanning sets and dimensions of the first cohomology groups of these algebras are given as follows:
\[
\begin{aligned}
H^{1}(\mu_1, \mu_1) &= \operatorname{span}_{\mathbb{C}}\{\overline{\xi_{21}}, \overline{\xi_{31}}, \overline{\xi_{34}}, \overline{\xi_{44}}, \overline{\xi_{52}}, \overline{\xi_{55}}\} \\
H^{1}(\mu_2, \mu_2) &= \operatorname{span}_{\mathbb{C}}\{\overline{\xi_{11}}, \overline{\xi_{21}}, \overline{\xi_{31}}, \overline{\xi_{34}}, \overline{\xi_{52}}\} \\
H^{1}(\mu_3, \mu_3) &= \operatorname{span}_{\mathbb{C}}\{\overline{\xi_{11}}, \overline{\xi_{21}}, \overline{\xi_{33}}\} \\
H^{1}(\mu_4, \mu_4) &= \operatorname{span}_{\mathbb{C}}\{\overline{\xi_{21}}, \overline{\xi_{33}}\}, \\
H^{1}(\mu_5, \mu_5) &= \operatorname{span}_{\mathbb{C}}\{\overline{\xi_{32}}, \overline{\xi_{33}}, \overline{\xi_{35}},\overline{\xi_{41}}, \overline{\xi_{51}}, \overline{\xi_{54}}\} \\
H^{1}(\mu^{\alpha}_6, \mu^{\alpha}_6) &= \operatorname{span}_{\mathbb{C}}\{\overline{\xi_{24}}, \overline{\xi_{32}}, \overline{\xi_{35}}, \overline{\xi_{51}}, \overline{\xi_{54}}\} \\
H^{1}(\mu_7, \mu_7) &= \operatorname{span}_{\mathbb{C}}\{\overline{\xi_{21}}, \overline{\xi_{31}}, \overline{\xi_{34}}, \overline{\xi_{41}}, \overline{\xi_{44}}, \overline{\xi_{55}}\} \\
H^{1}(\mu_8, \mu_8) &= \operatorname{span}_{\mathbb{C}}\{\overline{\xi_{11}}, \overline{\xi_{21}}, \overline{\xi_{31}}, \overline{\xi_{34}}, \overline{\xi_{41}}\} \\
H^{1}(\mu_9, \mu_9) &= \operatorname{span}_{\mathbb{C}}\{\overline{\xi_{22}}, \overline{\xi_{32}}, \overline{\xi_{51}}, \overline{\xi_{54}}\} \\
H^{1}(\mu_{10}, \mu_{10}) &= \operatorname{span}_{\mathbb{C}}\{\overline{\xi_{31}}, \overline{\xi_{32}}, \overline{\xi_{34}}, \overline{\xi_{55}}\} \\
H^{1}(\mu_{11}, \mu_{11}) &= \operatorname{span}_{\mathbb{C}}\{\overline{\xi_{31}}, \overline{\xi_{32}}, \overline{\xi_{34}}\} \\
H^{1}(\mu_{12}, \mu_{12}) &= \operatorname{span}_{\mathbb{C}}\{\overline{\xi_{11}}, \overline{\xi_{32}}, \overline{\xi_{34}}, \overline{\xi_{55}}\} \\
H^{1}(\mu_{13}, \mu_{13}) &= \operatorname{span}_{\mathbb{C}}\{\overline{\xi_{11}}, \overline{\xi_{32}}, \overline{\xi_{34}}\} \\
H^{1}(\mu_{14}, \mu_{14}) &= \operatorname{span}_{\mathbb{C}}\{\overline{\xi_{32}}, \overline{\xi_{44}}\} \\
H^{1}(\mu_{15}, \mu_{15}) &= \operatorname{span}_{\mathbb{C}}\{\overline{\xi_{32}}\} \\
H^{1}(\mu_{16}, \mu_{16}) &= \operatorname{span}_{\mathbb{C}}\{\overline{\xi_{31}}, \overline{\xi_{32}},\overline{\xi_{34}},\overline{\xi_{35}}, \overline{\xi_{44}}, \overline{\xi_{51}}\} \\
H^{1}(\mu_{17}, \mu_{17}) &= \operatorname{span}_{\mathbb{C}}\{\overline{\xi_{32}}, \overline{\xi_{35}}, \overline{\xi_{51}}\} \\
H^{1}(\mu_{18}, \mu_{18}) &= \operatorname{span}_{\mathbb{C}}\{\overline{\xi_{31}}, \overline{\xi_{55}}\} \\
H^{1}(\mu_{19}, \mu_{19}) &= \operatorname{span}_{\mathbb{C}}\{ \overline{\xi_{32}}, \overline{\xi_{33}}\} \\
H^{1}(\mu_{20}, \mu_{20}) &= \operatorname{span}_{\mathbb{C}}\{\overline{\xi_{32}}, \overline{\xi_{35}}, \overline{\xi_{44}}, \overline{\xi_{51}}\} \\
H^{1}(\mu^{\alpha}_{21}, \mu^{\alpha}_{21}) &= \operatorname{span}_{\mathbb{C}}\{\overline{\xi_{35}}\} \\
H^{1}(\mu^{\alpha}_{22}, \mu^{\alpha}_{22}) &= \operatorname{span}_{\mathbb{C}}\{\overline{\xi_{41}}, \overline{\xi_{44}}\}.
\end{aligned}
\]
\end{cor}

\begin{proof}
The proof of Corollary \ref{cr2} follows the same reasoning as that of Corollary \ref{cr1}, with minor adjustments.
\end{proof}
}

\section*{Conclusion}
This work focuses on the applications of low-dimensional cohomology groups. Specifically, we studied the zeroth and first Hochschild cohomology groups of 5-dimensional complex nilpotent associative algebras. Our results show that the dimensions of the zeroth cohomology groups range from 1 to 5, while the dimensions of the first cohomology groups range from 1 to 8.

\section*{Acknowledgements}

The authors thank the anonymous referees for their valuable suggestions and comments.

\section*{Funding}
The authors declare that no funding was received to support this research.

\section*{Conflicts of Interest}
The authors declare that they have no conflicts of interest.\\
\cite{A,B,C,D,E,F,G,H}

\end{document}